\documentclass[12pt]{amsart}

\usepackage[margin=1.15in]{geometry}

\usepackage{amsmath,amscd,amssymb,amsfonts,latexsym}
\usepackage{mathrsfs}
\usepackage{color}
\usepackage{bm}
\usepackage[all, cmtip]{xy}

\usepackage{url,mathtools,amsmath}

\definecolor{hot}{RGB}{65,105,225}

\usepackage[
colorlinks=true, linkcolor=hot ,  citecolor=hot, urlcolor=hot]{hyperref}

\theoremstyle{plain}
\newtheorem{theorem}{Theorem}[section]

\newtheorem{corollary}[theorem]{Corollary}

\newtheorem{lemma}[theorem]{Lemma}

\theoremstyle{definition}
\newtheorem{definition}[theorem]{\sc Definition}
\newtheorem{example}[theorem]{\sc Example}
\newtheorem{remark}[theorem]{\sc Remark}
\newtheorem{question}[theorem]{\sc Question}

\numberwithin{equation}{section}

\newcommand\cC{\mathcal{C}}

\newcommand{\im}{{\rm{Im\hspace{2pt}}}}

\def\bN{\mathbb{N}}

\def\bK{\mathbb{K}}
\def\bR{\mathbb{R}}
\def\bC{\mathbb{C}}
\def\bP{\mathbb{P}}

\def\m{\setminus}

\newcommand{\fin}{\hspace*{\fill}$\square$}

\DeclareMathOperator{\Sing}{Sing}

\DeclareMathOperator{\diffeo}{diffeo}

\title[]{Bifurcations of polynomial maps with diffeomorphic and connected fibers}

\author{Cezar Joi\c ta}
\address{Institute of Mathematics of the Romanian Academy, P.O. Box 1-764,
 014700 Bucharest, Romania and Laboratoire Europ\' een Associ\'e  CNRS Franco-Roumain Math-Mode}
\email{Cezar.Joita@imar.ro}

\author{Mihai Tib\u ar}
\address{Univ. Lille, CNRS, UMR 8524 -- Laboratoire Paul Painlev\'e, F-59000 Lille,
France}  
\email{mihai-marius.tibar@univ-lille.fr}

\thanks{The authors acknowledge support from the project ``Singularities and Applications'' - CF 132/31.07.2023 funded by the European Union - NextGenerationEU - through Romania's National Recovery and Resilience Plan.}

\subjclass[2010]{}

\begin{document}


\begin{abstract}  
We show the existence of polynomial maps  which have a regular bifurcation value, while over a neighbourhood of  this value the fibres are connected and diffeomorphic.

\end{abstract}

\dedicatory{Dedicated to Professor Vasile Br\^inz\u anescu on his $80^{th}$ birthday.}

\maketitle



\section{Introduction}

Let $F :\bK^m  \to \bK^p$ be a polynomial map, where $m\ge p\ge 1$, and $\bK$ is either $\bR$ or $\bC$.  
We say that $F$ is \emph{(locally) ${\cC} ^\infty$ trivial at $t_0\in \bK^p$} if there is a
neighbourhood $D$ of $t_0$ such that the
restriction $F_| :  F^{-1} (D) \to D$ is a ${\cC} ^\infty$ trivial
fibration, i.e. there exists a commutative diagram
$$\xymatrix{
& F^{-1}(D) \ar[dr]^{F} \ar[rr]^{\sim}&& F^{-1}(t_0)\times D\ar[dl]^{p_2}\\
 && D 
}$$
 We say that $t_0$ is a bifurcation value of $F$ whenever $F$ is not ${\cC} ^\infty$ trivial at $t_0$. The set of bifurcation values of $F$ is denoted by $B(F)$.
If $\Sing(F)$ denotes the set of points in $\bK^m$ where $F$ is not a submersion, 
then, obviously, $B(F)\supset \overline{F(\Sing(F))}$.

If $m=p$, the famous Jacobian Conjecture asks to show that $\Sing(F)=\emptyset$ implies that $B(F)=\emptyset$.
If $\bK=\bR$, the Jacobian Conjecture is known to be false, as shown by Pinchuk \cite{Pi}. In the complex case, the 
conjecture is wide open.
\medskip

We will assume from now on that $m>p$.
\medskip

Characterising  the bifurcation locus of a map $F$ is a difficult problem. As the following example shows, the condition  $\Sing(F)\neq\emptyset$  cannot be the only obstruction to local triviality.
\begin{example} \cite{Br} \label{ex:1}
 Let $F: \bK^2\to \bK, \ \ F(x,y) =  x + x^2y$. Then  $\Sing(F)=\emptyset$, while $0$ is a bifurcation value.
 \end{example}
 To understand the bifurcation locus $B(F)$ one has to study the asymptotical behavior of $F$ ``at infinity". One way to do it, is to compactify the fibers of
 $F$ in $\bP^n$, and to control the intersections of the compactified fibers with the hyperplane at infinity. In the case $p=1$,  one may find in \cite{Ti2} several methods of control under some natural conditions. In the case $p\geq 2$, several evaluations of $B(F)$ may be found in e.g. \cite{Ti1}, \cite{Ga}, \cite{KOS},  \cite{Je}, \cite{DRT}.
 
 \medskip
 
A necessary condition for the map $F$ to be locally trivial at $t_0$ is that all fibers of $F$ are diffeomorphic in a neighborhood of $t_0$.
 The following theorem due to Suzuki \cite{Su}, see also \cite{HL}, shows that, in the special case of polynomial functions $\bC^2\to\bC$, this condition is also sufficient.

\begin{theorem}
Let $F: \bC^2 \to \bC$ be a polynomial function and let $t_0 \in \bC\m F(\Sing F)$. Then $t_0$ is not a bifurcation value if and only if the Euler characteristic of the fibres $\chi(F^{-1}(t))$ is constant for $t$ varying in some neighborhood of $t_0$.
\end{theorem}

It was discovered in \cite{TZ} that for polynomial functions $\bR^2\to\bR$, the constancy of  the Euler characteristic of the fibres is not sufficient for local triviality, and that this is due to two new phenomena which might appear: ``vanishing" and ``splitting".  Each of these phenomena necessarily produce bifurcation values. For instance in Example \ref{ex:1} one has a splitting at infinity in the fibre $F^{-1}(t)$ when $t\to 0$.

Here is an example from \cite{TZ} where a bifurcation value occurs despite the fact that the fibres of $F$ are diffeomorphic.
This shows that the vanishing and the splitting may occur in pairs, and such that the diffeomorphism type of the fibre is preserved.
\begin{example} \label{ex:eulercst}
The two variable polynomial $F:\bR^2 \to \bR$,
\[  F(x,y):=x^2y^3(y^2-25)^2+2xy(y^2-25)(y+25)-
(y^4+y^3-50y^2-51y+575)  \]
has the property that for all $t\in \bR$ such that $|t|$ is small enough, the fibre $F^{-1}(t)$  has $5$ non-compact connected components, thus each component is diffeomorphic to a line. Therefore the fibres $F^{-1}(t)$ are diffeomorphic, for all small enough $|t|$.  Nevertheless, it turns out that the value 0 is a bifurcation value. This is due to the fact that, when $t<0$ and $t\to 0$,  out of the 5 lines components of the fibre $F^{-1}(t)$, two split at infinity, and two others vanish at infinity.
\end{example}

The results of \cite{TZ} were extended in \cite{JT1} to polynomial maps $\bR^{n+1}\to\bR^n$. We present below the definition of vanishing as it was formulated in \cite{JT1} together with one of the results of that paper.

\begin{definition} Let $F:M\to B$ be a $\cC^\infty$ map between $\cC^\infty$ manifolds, and let $\lambda$ be a point in $B$.
 We say that there are \emph{vanishing components at infinity when $t$ converges to $\lambda$} if there 
 is a sequence of points $t_k \in B$,  $t_k\to \lambda$,  such that for some choice of a connected component $C_{t_k}$ 
 of $F^{-1}(t_k)$ the sequence of sets $\{C_{t_k}\}_{k\in \bN}$ is \emph{locally finite}, i.e., for any compact $K\subset M$, 
 there is an integer $p_K\in \bN$ such that $C_{t_q}\cap K = \emptyset$ for all $q \ge p_K$.
 \end{definition}
 
 \begin{theorem}[\cite{TZ} for $n=2$; \cite{JT1} for $n\geq 3$] \label{t:real}
 Let $X\subset \bR^m$ be a real nonsingular irreducible algebraic set of dimension $n\ge 2$ and  let  $F : X\to \bR^{n-1}$ be an algebraic map. Let $a$ be an interior point of the set $\im F \setminus \overline{F(\Sing F)}\subset \bR^{n-1}$ and let $X_{t} := F^{-1}(t)$.  Then $a\not\in B(F)$ 
if and only if the following two conditions are satisfied:
\begin{enumerate}
 \item 
 the Euler characteristic $\chi(X_{t})$ is constant for
$t$ in some neighbourhood of $a$, 

\item there is no  
component of $X_t$ which vanishes at infinity as $t$ tends to $a$.
\end{enumerate}\fin
\end{theorem}

In the case of complex polynomial maps $\bC^{n+1}\to\bC^n$, we have shown in \cite{JT2} that when $n\geq 2$, unlike the case $n=1$,
one might have a similar vanishing phenomenon.

\begin{example}\cite{JT2}
Let $F:\bC^3\to\bC^2$ be defined by $$F(x,y,z)=\left(x,[(x-1)(xz+y^2)+1][x(xz+y^2)-1]\right).$$
One easily checks that $(0,0)$ is an interior point of $\bC^2\setminus \Sing F$. It turns out that all fibers of $F$ in a small neighborhood of $(0,0)\in\bC^2$  are isomorphic to
$ \bC \sqcup \bC$, and that $(0,0)$ is a bifurcation point.
\end{example}

On the positive side, we have proved the following theorem

\begin{theorem}\cite{JT2}\label{t:complex}
 Let $F:M\to B$ be a  holomorphic map between connected complex manifolds, where $M$ is Stein and $\dim M=\dim B+1$.
We asume that the Betti numbers  $b_0(t)$ and $b_1(t)$ of all the fibers  $F^{-1}(t)$  are finite.
 Let $\lambda \in \im F \setminus \overline{F(\Sing F)}\subset B$. Then  $\lambda\not\in B(F)$ 
if and only if the following two conditions are satisfied:
\begin{enumerate}
 \item 
the Euler characteristic of the fibres is constant for $t$ in some neighborhood of $\lambda$
\item no connected component of $F^{-1}(t)$ is vanishing at infinity when $t\to\lambda$.
\end{enumerate} \fin
\end{theorem}

By Theorem \ref{t:real} and \ref{t:complex}, we get:

\begin{corollary}\label{c:complex}
Let $F:\bK^{n+1}\to\bK^n$ be polynomial map and $t_0\in  \im F \setminus \overline{F(\Sing F)}$. If the fibers $F^{-1}(t)$ are connected and diffeomorphic for $t$ in some neighborhood of $t_0$, then $F$ is locally trivial at $t_0$.
\fin
\end{corollary}

 Corollary \ref{c:complex} settles the case of 1-dimensional fibres, which was originally a question by Gurjar in the holomorphic setting, see \cite{JT2}. 

Gurjar's question for higher dimensional fibres sounds as follows:

\begin{question}\label{q:main}
If $F:\bK^{m}\to\bK^p$ is a polynomial map with  $m>p+1$, and $t_0\in  \im F \setminus \overline{F(\Sing F)}$ is such that  the fibers
of $F$ are connected and diffeomorphic over some neighborhood of $t_0$, does it follow that $F$ is locally trivial at $t_0$?
\end{question}

To this day, this question is open in general.   Its answer is positive in the following special case. 
\begin{theorem}\cite{Pa}, \cite{Me}.
If $F:M\to B$ is a  smooth map between $\mathcal{C}^\infty$ manifolds and all fibers of $F$ are diffeomorphic to $\bR^k$ ($k=\dim M-\dim B$),
then $F$ is a locally trivial fibration.
\fin
\end{theorem}

The above result was proved by Palmeira \cite{Pa} for maps with 1-dimensional target, and extended by Meignez \cite{Me} 
to maps with arbitrary dimensional target.

\

In this note we show that the answer to Question \ref{q:main} is negative in general. 

We will treat separately the real and the complex cases. 

\section{Special bifurcation values, real setting}\label{s:real}

\begin{theorem} \label{t:exreal}
There exists a polynomial function $F:\bR^4\to\bR^2$, and a point $t_0\in\bR^2$
such that:
\begin{enumerate}
\item $t_0$ is an interior point of $\im F\setminus F(\Sing(F))$,  
\item all the fibers of $F$ on a neighborhood
of $t_0$ are connected and diffeomorphic,  
\item $t_0\in B(F)$.
\end{enumerate}
\end{theorem}

\begin{proof} 
We start with a polynomial function $f:\bR^2\to \bR$ such that $f(z,\lambda)\geq 0$ for all $(z,\lambda)\in\bR^2$, and later we will make a more precise choice of $f$.
For the moment we define the polynomial map  $F:\bR^4\to\bR^2$ as:  
$$F(x,y,z,\lambda)=\bigl(x^2-y+y^2f(z,\lambda),\lambda\bigr) .$$
We will use the notation $U:=]0,\infty[\times \bR\subset \bR^2$. 

\medskip

\noindent
\textbf{Claim 1.}  $F$ is a submersion on $F^{-1}(U)$. This means that we have to show that if $F(x,y,z,\lambda)\in U$, then
the following system 
$$\left\{
\begin{array}{l}
2x=0\cr
-1+2yf(z,\lambda)=0\cr
y^2\frac{\partial f}{\partial z}(z,\lambda)=0,
\end{array}
\right.$$
has no solution.

Suppose that this system has a solution $(x,y,z,\lambda)$. From $-1+2yf(z,\lambda)=0$ we get that $y\not= 0$ and $f(z,\lambda)=\frac1{2y} \ge 0$, thus $y>0$. We  also have $x=0$, and therefore $F(x,y,z,\lambda)=(-\frac y2,\lambda)$.
 We deduce that $F(x,y,z,\lambda)\not \in U$, hence the system does not have solutions in $U$, and consequently  our claim is proved.
 
\medskip

We will show that for some fixed point $(b,\lambda)\in U$, we have the diffeomorphism:
\begin{equation}\label{eq:diffeo}
 F^{-1}(b,\lambda)  \stackrel{\diffeo}{\simeq}
K_\lambda 
\end{equation}
for any $\lambda\in \bR$, where
$$K_\lambda := \bigl\{(u,v,a)\in\bR^3:u^2+v^2=1\bigr\} \setminus \bigl\{(0,1,a)\in \bR^3 : f(a,\lambda)=0\bigr\},$$

by constructing two reciprocal functions.


\medskip

\noindent
\textbf{Claim 2.} 
The function $g_{b,\lambda}: F^{-1}(b,\lambda)\to\bR^3$, 
$$g_{b,\lambda}(x,y,z,\lambda)=\left(\frac x{\sqrt{x^2+y^2}}, \frac y{\sqrt{x^2+y^2}},z\right),$$
is a well-defined, $\mathcal{C}^\infty$-function. Indeed, 
if $(x,y,z,\lambda)\in F^{-1}(b,\lambda)$, then $x^2-y+y^2f(z,\lambda)=b$. Since $(b,\lambda)\in U$, we must have $b\neq 0$, 
which implies that $(x,y)\neq (0,0)$.  

\medskip

\noindent
\textbf{Claim 3.} 
The intersection $g_{b,\lambda}(F^{-1}(b,\lambda))\bigcap \bigl\{(0,1,a): f(a,\lambda)=0 \bigr\}$ is empty.\\
Indeed, 
 if $\left(\frac x{\sqrt{x^2+y^2}}, \frac y{\sqrt{x^2+y^2}},z\right)=(0,1,a)$
with $f(a,\lambda)=0$, then $x=0$, $\frac y{\sqrt{y^2}}=1$ and hence $y>0$, $z=a$.
Since $(x,y,a,\lambda)\in F^{-1}(b,\lambda)$, we have $x^2-y+y^2f(a,\lambda)=b$, thus we get $b<0$, which contradicts the fact that $(b,\lambda)\in U$.

\medskip

By the definition of $g_{b,\lambda}$ we have the inclusion 
$g_{b,\lambda}(F^{-1}(b,\lambda))\subset \{(u,v,a)\in\bR^3:u^2+v^2=1\}$, thus by Claim 3 we have
got a more precise target of our $\mathcal{C}^\infty$-function, namely $\im g_{b,\lambda} \subset K_\lambda$.

\medskip

\noindent
\textbf{Claim 4.} For any fixed $(u,v,a)\in K_\lambda$, and $(b,\lambda)\in U$, 
the equation in $\xi$:
\begin{equation}\label{eq:xi}
(u^2+v^2f(a,\lambda)) \xi^2- v\xi -b=0
\end{equation}
 has precisely one positive solution.

If $u^2+v^2f(a,\lambda)\neq 0$ then \eqref{eq:xi} is a quadratic equation with real solutions,  the product 
of its roots is $\frac{-b}{u^2+v^2f(a,\lambda)}<0$ (since $b>0$ by our assumption), thus one of its roots is positive and the other is negative. 

If $u^2+v^2f(a,\lambda)=0$, then we must have
$u=0$ and $vf(a,\lambda)=0$. Since $(u,v,a)\in  K_\lambda$, thus $u^2+v^2=1$, we get $v= \pm 1$, hence $v\neq 0$ and thus $f(a,\lambda)=0$, so we actually must have $v=-1$. The solution of \eqref{eq:xi} is then $\xi=b$, which is positive. This finishes the proof of Claim 4.

\medskip

We still keep our $(b,\lambda)\in U$ fixed, and define now the function  $\alpha : K_\lambda \to\bR$ which associates to $(u,v,a)\in  K_\lambda$  the unique positive solution of \eqref{eq:xi}.

\

\noindent
\textbf{Claim 5.}  The function $\alpha$ is  $\mathcal{C}^\infty$.

Indeed, as solution of \eqref{eq:xi}, $\alpha$ can be written as:
\begin{equation}\label{eq:alpha}
\alpha(u,v,a)=\frac{v+\sqrt{v^2+4b(u^2+v^2f(a,\lambda))}}{2(u^2+v^2f(a,\lambda))}
\end{equation}
whenever $u^2+v^2f(a,\lambda)\neq 0$, thus $\alpha$ is of class $\mathcal{C}^\infty$ 
on an open subset of  $K_\lambda$.

On the complement of this open subset of $K_\lambda$, i.e. at some point where $u^2+v^2f(a,\lambda)= 0$, we have seen in Claim 4 above that this point
is of the form $(0,-1,a)$,  where $f(a,\lambda)=0$. Notice that the fraction at \eqref{eq:alpha} is equal to 
\begin{equation}\label{eq:alpha2}
 \alpha(u,v,a)=\frac{2b}{\sqrt{v^2+4b(u^2+v^2f(a,\lambda))}-v}
\end{equation}
whenever both fractions are well-defined.
Since the denominator of  the fraction at \eqref{eq:alpha2} is strictly positive at $(0,-1,a)$, with $f(a,\lambda)=0$, it follows that the function $\alpha$ is well-defined and of class $\mathcal{C}^\infty$ in some neighborhood of any such point.
This finishes the proof of our claim.

\

Using now the function $\alpha = \alpha (u,v,a)$, we define the function:  
$$h_{b,\lambda}: K_\lambda \to\bR^4, \ \ 
h_{b,\lambda}(u,v,a)=(u\alpha ,v \alpha,a,\lambda).$$
which is $\mathcal{C}^\infty$ by its definition. We observe that its image is in $F^{-1}(b,\lambda)$,
since
 $$F(u\alpha ,v\alpha, a,\lambda)=\bigl(u^2\alpha^2-v \alpha +v^2f(a,\lambda)\alpha^2,\lambda\bigr)=(b,\lambda).$$ 
 
 \
 
\noindent
\textbf{Conclusion.} $F^{-1}(b,\lambda) \stackrel{\diffeo}{\simeq} K_\lambda $.

Both sets are of dimension 2, and it turns out  by direct computation that the $\mathcal{C}^\infty$ function 
$$h_{b,\lambda} : K_\lambda \to F^{-1}(b,\lambda)$$
 is the inverse of the $\mathcal{C}^\infty$ function
 $g_{b,\lambda} : F^{-1}(b,\lambda) \to K_\lambda $.\\
   This finishes the proof of \eqref{eq:diffeo}.
 
\medskip

\noindent
\textbf{Particular choice.} We now make a particular choice for the polynomial $f$. Let us set:
$$f: \bR^2 \to \bR, \ \ 
f(z,\lambda)=(z^2+\lambda^2)(\lambda z-1)^2.$$
Then $f(z,\lambda)\geq 0$, and the set $\{ a\in \bR : f(a,\lambda)= 0\}$ consists of a single point, namely the point $\frac1\lambda$ in case $\lambda \not=0$,  and the point 0 in case $\lambda =0$.  Consequently, the set $K_\lambda$ 
is  diffeomorphic to the infinite cylinder with one point removed, which is the diffeomorphic to the sphere $S^2$ with 3 points removed.

\medskip
 
\noindent
\textbf{End of the proof: $F$ is not a locally trivial fibration at $(1,0)\in \bR^2$.} 

Let us consider the loop:
$$\gamma:[0,2\pi]\to \bR^3 , \ \ \gamma(\theta)=(\cos\theta,\sin\theta,1).$$

For $\lambda\in ]-1, 0[ \cup ]0,1[$, the image of
$\gamma$ is included in the set 
$$K_\lambda = \{(u,v,a)\in\bR^3:u^2+v^2=1\}\setminus\{(0,1,\frac 1\lambda)\},$$
and for $\lambda =0$,  the image of $\gamma$ is included in the set 
$$K_0= \{(u,v,a)\in\bR^3:u^2+v^2=1\}\setminus\{(0,1,0)\}.$$

 Therefore, for any $\lambda\in ]-1, 1[$, we can compose $\gamma$ by $h_{1,\lambda}$ and obtain the function:
$$\gamma_{\lambda}:[0,2\pi]\to F^{-1}(1,\lambda),\ \ \  \gamma_{\lambda}=h_{1,\lambda}\circ \gamma.$$  
More precisely, we have:
$$\gamma_{\lambda}(\theta)=(\alpha_{\lambda}(\theta)\cos\theta, \alpha_{\lambda}(\theta)\sin\theta,1,\lambda)$$ 
where $\alpha_{\lambda}(\theta)$
is the unique positive solution of the equation \eqref{eq:xi}, which becomes in our case:
\begin{equation}\label{eq:sol}
 \xi^2\left(\cos^2\theta+\sin^2\theta(1+\lambda^2)(\lambda-1)^2\right)-\xi\cos\theta-1=0.
\end{equation}

It follows immediately from the quadratic formula for the solutions of \eqref{eq:sol} that, when $\lambda\to 0$, the solution $\alpha_{\lambda}(\theta)$, as a function of $\theta$, converges uniformly to the positive solution of the equation:
$$\xi^2\left(\cos^2\theta+\sin^2\theta\right)-\xi\cos\theta-1=0.$$

This implies that the loop $\gamma_\lambda$ converges uniformly, when $\lambda\to 0$, to the loop $\gamma_0 := h_{1,0}\circ \gamma$. 

\

We now consider the following second loop:
$$\tilde\gamma(\theta)=(\cos\theta,\sin\theta,-1) \  \mbox{ and } \
\tilde\gamma_\lambda :=h_{1,\lambda}\circ \tilde\gamma.$$
By the same arguments as above,  we deduce that $\tilde\gamma_\lambda$ converges uniformly, as $\lambda\to 0$, to $\tilde\gamma_0 := h_{1,0}\circ \tilde\gamma$. 

\sloppy
\hyphenation{ho-mo-to-pi-ca-lly}

We then have:

\noindent (a).  For  $\lambda\not= 0$,  $\lambda\to 0$, the loops  $\gamma_\lambda$ and $\tilde\gamma_\lambda$ are homotopically equivalent in $F^{-1}(1,\lambda)$, since 
$\gamma$ and $\tilde\gamma$ are homotopically equivalent in $K_\lambda$. 

\noindent  (b). $\gamma_0$ and $\tilde\gamma_0$ are not homotopically  equivalent in $F^{-1}(1,0)$, 
since 
$\gamma$ and $\tilde\gamma$ 
 are not  homotopically equivalent in $K_0$. 
 
 \
 
Let us now finally  show that $F$ is not locally trivial at $(1,0)$.
By \emph{reductio ad absurdum}, if  $F$ is a trivial fibration at $(1,0)\in \bR^2$,
 we consider a local trivialization $G$ of $F$ over some neighbourhood $V\subset \bR^2$ of $(1,0)$. Then, the application of the following lemma to $G$ and to the homotopy equivalences of loops which correspond to (a) and (b) listed just above, yields a contradiction. 

This ends the proof of our theorem.
\end{proof}

\begin{lemma}\label{lemma}
Let  $W$ and $V$ be smooth manifolds, let $G:W\times V\to V$ be the projection map, let $t_0\in V$, and let $\{t_n\}$ be a sequence of points in $V$ which converges to $t_0$. 
Let $\gamma_n, \tilde\gamma_n:S^1\to G^{-1}(t_n)\subset W\times V$, and  $\gamma_0,\tilde\gamma_0:S^1\to G^{-1}(t_0)\subset W\times V$ be continuous maps such that
$\gamma_n$ converges uniformly to $\gamma_0$, and $\tilde\gamma_n$ converges uniformly to $\tilde\gamma_0$.

 If $\gamma_n$ and $\tilde\gamma_n$ are 
homotopically equivalent in $G^{-1}(t_n)$ for every $n\geq 1$,  then $\gamma_0$ and $\tilde\gamma_0$ are 
homotopically equivalent in $G^{-1}(t_0)$.
\end{lemma}

\begin{proof}[Sketch of the proof]
 Since $G$ is the projection map, we may write, for $n\geq 0$, $\gamma_n(\lambda)=(\mu_n(\lambda),t_n)$ and 
$\tilde\gamma_n(\lambda)=(\tilde\mu_n(\lambda),t_n)$ where
$\mu_n,\tilde\mu_n:S^1\to W$ are continuous maps such that $\mu_n$ converges uniformly to $\mu_0$ and 
$\tilde\mu_n$ converges uniformly to $\tilde\mu_0$. 

In order to prove our lemma, it is sufficient to prove that, for $n$ large enough, $\mu_n$ is homotopically equivalent to $\mu_0$.
To show this claim, we first cover the image of $\mu_0$ with finitely many open sets diffeomorphic to balls. We view $\mu_n$ and $\mu_0$ as being defined on
$[0,1]$ with $\mu_n(0)=\mu_n(1)$, $\mu_0(0)=\mu_0(1)$. By the uniform convergence, there exists $m$ large enough and  a partition of the unit segment by $0=\lambda_1\leq\lambda_2\leq\cdots\leq\lambda_{k_m}=1$ together with an open covering of the image $\im \mu_0$ by connected open sets $B_1, \ldots , B_{k_m} \subset W$ such that, for any $n\ge m$, 
$\mu_n([\lambda_j,\lambda_{j+1}])$ and $\mu_0([\lambda_j,\lambda_{j+1}])$ are both included in 
$B_j$,  and $\mu_0(\lambda_j)$ and $\mu_n(\lambda_j)$ are in  same connected component of  $B_j\cap B_{j-1}$.  Let $\theta_j(s)$ be a path  in $B_j \cap B_{j-1}$  joining
 $\mu_0(\lambda_j)$ to $\mu_n(\lambda_j)$. 
One may construct a continuous map $h_j:[0,1]\times[0,1]\to B_j$ such that
$h_j(s,0)=\theta_j(s)$, $h_j(s,1)=\theta_{j+1}(s)$, $h_j(0,t)=\mu_0(t)$, $h_j(1,t)=\mu_n(t)$. The homotopy between $\mu_0$ and $\mu_n$ is obtained by piecewise glueing these homotopies $h_j$. This ends the proof of our claim.
\end{proof}

\begin{remark}\label{r:real}
The space $M :=F^{-1}\Big(\{1\}\times \bR\Big)$ is a 3-fold in $\bR^4$, and $F$ restrict to an algebraic function $F_|:M\to\bR$ the fibers of which are the same as those of $F$. The proof of the above Theorem shows that  they are diffeomorphic and connected, and that the restriction $F_|$ is not locally trivial at $0$.
\end{remark}


\section{Special bifurcation values, complex setting}\label{s:complex}

In the complex setting we have:

\begin{theorem}\label{t:excomplex}
There exists a polynomial function $F:\bC^4\to\bC^2$, and a point $t_0\in\bC^2$
such that:
\begin{enumerate}
\item $t_0$ is an interior point of $\im F\setminus F(\Sing(F))$, 
\item all the fibers of $F$ on a neighborhood
of $t_0$ are connected and diffeomorphic,
\item $t_0\in B(F)$.
\end{enumerate}
\end{theorem}

\begin{proof} Let $f:\bC^4\to\bC$ be defined by 
$$f(x,y,z,\lambda)=x(y^2+\lambda z-1)(\lambda y^2+ \lambda^2 z-\lambda+1),$$ 
and let $F:\bC^4\to\bC^2$,  
$$F(x,y,z,\lambda)=(f,\lambda).$$

We fix the point $t_0=(1,0)$ and we denote by $U$ an open neighborhood of $t_0$ which does not intersect $\{0\}\times \bC$.
We have $\frac{\partial f}{\partial x}=(y^2+\lambda z-1)(\lambda y^2+
\lambda^2 z-\lambda+1)$,  thus $\frac{\partial f}{\partial x}=0$  implies
$f=0$. This shows that  $F$ is submersion on $F^{-1}(U)$.

We will show now that for each $(b,\lambda)\in U$ the fiber $F^{-1}(b,\lambda)$ is algebraically isomorphic to 
$\bC\times (\bC\setminus\{-1,1\})$.

 \medskip
 
\noindent
{\bf Case 1}. $\lambda =0$. \\ We have
 $F^{-1}(b,0)=\{(x,y,z)\in\bC^3:x(y^2-1)=b\}\times\{0\}\subset \bC^4$. 
As $b\neq 0$, by separating $x$, we immediately get an isomorphism $F^{-1}(b,0) \simeq \bC\times (\bC\setminus\{-1,1\})$.

\medskip

\noindent {\bf Case 2}. $\lambda \neq0$. \\ We then have:
\begin{align*}
F^{-1}(b,\lambda)=\Big\{(x,y,z)\in\bC^3:
x(y^2+\lambda z-1)( y^2+
\lambda z-1+\frac 1\lambda)=\frac b\lambda\Big\}\times\{\lambda\}.
\end{align*}

Let then $h_\lambda:\bC\times(\bC\setminus\{0,-\frac 1\lambda\})\to F^{-1}(b,\lambda)$ be defined by:
$$h_\lambda(u,v)=\Big(\frac b{v(\lambda v+1)},u,\frac{v-u^2+1}\lambda,\lambda\Big).$$
 
It is straightforward to check that
$$g_\lambda(x,y,z,\lambda)=(y,y^2+\lambda z-1)$$
is its inverse. We thus have our claimed isomorphism in this case too.

Let us prove now that $F$ is not locally trivial at $(1,0)$.  To do that, we will construct two loops in 
$F^{-1}(b,\lambda)$ that are homotopy equivalent, then  make $\lambda \to 0$, and obtain 
their limits, which turn out to be homotopy non-equivalent loops in $F^{-1}(b,0)$.  

 We consider the following two loops:
\begin{align*}
&\gamma_{1,\lambda}:[0,2\pi]\to\bC\times\Big(\bC\setminus\{0,-
\frac 1\lambda\}\Big),\ \ \ \gamma_{1,\lambda}(\theta)=(1+e^{i\theta},(1+e^{i\theta})^2-1)\\
&\gamma_{2,\lambda}:[0,2\pi]\to\bC\times\Big(\bC\setminus\{0,-
\frac 1\lambda\}\Big),\ \ \ \gamma_{2,\lambda}(\theta)=(-1+e^{i\theta},(-1+e^{i\theta})^2-1).
\end{align*}

The paths $\theta\mapsto(1+e^{i\theta})^2-1=e^{i\theta}(e^{i\theta}+2)$
and $\theta\mapsto(-1+e^{i\theta})^2-1=e^{i\theta}(e^{i\theta}-2)$ have the same winding number around $0$.  It then follows that,  for $\lambda\not=0$ close enough to $0$, the loops
$\gamma_{1,\lambda}$ and $\gamma_{2,\lambda}$ are homotopically equivalent in 
$\bC\times(\bC\setminus\{0,-\frac 1\lambda\})$. Therefore $h_\lambda\circ\gamma_{1,\lambda}$
and $h_\lambda\circ\gamma_{2,\lambda}$ are homotopically equivalent in  $F^{-1}(b,\lambda)$, for $\lambda \not= 0$ and $b\not= 0$. 
By composing with  $h_\lambda$, and then making $\lambda\to 0$, we get:
 $$\lim_{\lambda \to 0}(h_\lambda\circ\gamma_{1,\lambda})(\theta)= \Bigl(\frac{b}{e^{i\theta}(e^{i\theta}+2)},1+e^{i\theta},0,0\Bigr),$$ 
$$\lim_{\lambda \to 0}(h_\lambda\circ\gamma_{2,\lambda})(\theta)=\Bigl(\frac{b}{e^{i\theta}(e^{i\theta}-2)},-1+e^{i\theta},0,0\Bigr).$$

As one can easily check, the two limits displayed above, call them $\gamma_{1,0}$ and $\gamma_{1,0}$, are loops in  the fibre 
$$F^{-1}(b,0) := \bigl\{(\frac{b}{y^2-1}, y, z, 0)\in\bC^4 : y \not= \pm 1\bigr\},$$
where $b\not= 0$, and $y\in \bC\setminus\{-1,1\})$.

If the loops $\gamma_{1,0}$ and $\gamma_{1,0}$ are homotopy equivalent  in $F^{-1}(b,0)$, then so must be their projections to the second factor  $\pi_2: F^{-1}(b,0) \to \bC\setminus\{-1,1\}$, namely
$\theta \mapsto (1+e^{i\theta})$ and $\theta \mapsto (-1+e^{i\theta})$.
But it is easy to check that these two later loops are \emph{not} homotopically equivalent in $\bC\setminus\{-1,1\}$.  This contradiction shows that the two limit loops $\gamma_{1,0}$ and $\gamma_{1,0}$ are not homotopy equivalent  in $F^{-1}(b,0)$,
We may then apply Lemma \ref{lemma}, as we have done precedingly, to conclude that $F$ is not locally trivial at $(1,0)$.
\end{proof}

The following remark is similar to Remark \ref{r:real}:

\begin{remark}\label{r:complex}
The complex algebraic map $F$ restricts to an algebraic function $F_|:M\to\bC$, where $M=F^{-1}\Big(\{1\}\times \bC\Big)$ is a complex 3-fold in $\bC^4$, the fibers of which are the same as those of $F$ over $\{1\}\times \bC$. The proof of Theorem \ref{t:excomplex} shows that they are connected  and diffeomorphic over some neighbourhood of the point $0\in \bC$, and also shows that  the restriction $F_|$ is not locally trivial at $0$.
\end{remark}

\

 Given our Theorems \ref{t:exreal} and \ref{t:excomplex}, and our  Remarks \ref{r:real} and \ref{r:complex}, respectively,  the following question on the bifurcation set arises naturally:

\begin{question}
Let $F:\bK^3\to\bK$ be a  polynomial function such that its fibres over some neighborhood of  $t_0\in \im F\setminus F(\Sing(F))$ are connected 
and diffeomorphic. Is then $t_0$ a non-bifurcation value of $F$?
\end{question} 
 



\end{document}